\documentclass[12pt]{article}
\usepackage{amssymb,amsmath}
\usepackage{amsthm}
\usepackage{mathtools}
\usepackage{cases}
\usepackage{amsfonts}
\usepackage{color,xcolor}
\usepackage[left=2.0cm,right=2.0cm,top=2.0cm,bottom=2.0cm]{geometry}
\usepackage[colorlinks,citecolor=blue,urlcolor=blue]{hyperref}
\usepackage[utf8]{inputenc}
\usepackage{pdfsync}

\newtheorem{theorem}{Theorem}[section]

\newtheorem{lemma}{Lemma}[section]

\newtheorem{definition}{Definition}[section]
\newtheorem{remark}{Remark}[section]

\newcommand{\bal}{\begin{align}}
\newcommand{\bbal}{\begin{align*}}
\newcommand{\beq}{\begin{equation}}
\newcommand{\eeq}{\end{equation}}
\newcommand{\bca}{\begin{cases}}
\newcommand{\eca}{\end{cases}}

\newcommand{\pa}{\partial}
\newcommand{\fr}{\frac}

\newcommand{\De}{\Delta}

\newcommand{\dd}{\mathrm{d}}

\newcommand{\R}{\mathbb{R}}

\newcommand{\g}{\big}

\begin{document}
\title{Ill-posedness for the two component  Degasperis-Procesi equation in critical Besov space}

\author{Jinlu Li$^{1}$, Min Li$^{2,}$\footnote{E-mail: lijinlu@gnnu.edu.cn; limin@jxufe.edu.cn(Corresponding author); mathzwp2010@163.com} and Weipeng Zhu$^{3}$\\
\small $^1$ School of Mathematics and Computer Sciences, Gannan Normal University, Ganzhou 341000, China\\
\small $^2$ Department of mathematics, Jiangxi University Of Finance and Economics, Nanchang 330032, China\\
\small $^3$ School of Mathematics and Big Data, Foshan University, Foshan, Guangdong 528000, China}

\date{\today}

\maketitle\noindent{\hrulefill}

{\bf Abstract:}
In this paper, we study the Cauchy problem for the two component Degasperis-Procesi equation in critical Besov space $B^1_{\infty,1}(\mathbb R)$. By presenting a new construction of initial data, we proved the norm inflation of the corresponding solutions in $B^1_{\infty,1}(\mathbb R)$ and hence ill-posedness. This is quite different from the local well-posedness result \cite{LYZ} for the Degasperis–Procesi equation in critical Besov space $B^1_{\infty,1}(\mathbb R)$ due to the coupled structure of density function.


\vskip0mm\noindent{\hrulefill}

\section{Introduction}\label{sec1}
In this paper, we consider the Cauchy problem for the following two component Degasperis-Procesi equation \cite{POP,YY}
\begin{align}\label{dp0}
\begin{cases}
\pa_tm+3mu_x+m_xu+k_3\rho\rho_x=0, &\quad (t,x)\in \R^+\times\R,\\
\pa_t\rho+k_2u\rho_x+(k_1+k_2)u_x\rho=0,&\quad (t,x)\in \R^+\times\R,\\
m=u-u_{xx},&\quad (t,x)\in \R^+\times\R,\\
u(0,x)=u_0(x), \quad \rho(0,x)=\rho_0(x), &\quad x\in \R,
\end{cases}
\end{align}
here $(k_1,k_2,k_3)=(1,1,c)$ or $(c,1,0)$ and $c$ is an arbitrary constant. System (\ref*{dp0}) first appeared in \cite{POP} as the Hamiltonian extension of the Degasperis–Procesi equation. It is worth mentioning that, in the  case $(k_1,k_2,k_3)=(c,1,0),$ system (\ref{dp0}) is no more coupled and the equation on $\rho$ becomes linear. Therefore, we only consider the first case $(k_1,k_2,k_3)=(1,1,c).$ For simplicity, we might as well assume that $c=1$, which makes system (\ref{dp0}) equivalent to the following quasi-linear evolution equation of hyperbolic type:
\begin{align}\label{dp}
    \begin{cases}
    \pa_tu+uu_x=-\pa_x(1-\pa^2_x)^{-1}(\frac32u^2+\frac 12\rho^2), &\quad (t,x)\in \R^+\times\R,\\
    \pa_t\rho+u\rho_x=-2u_x\rho,&\quad (t,x)\in \R^+\times\R,\\
    u(0,x)=u_0(x), \quad \rho(0,x)=\rho_0(x), &\quad x\in \R.
    \end{cases}
    \end{align}
In particular, for $\rho_0\equiv 0,$ system (\ref*{dp}) reduces to the Degasperis–Procesi (DP) equation\cite{DP}. The  DP  equation  can  be  regarded  as  a  model  for  nonlinear  shallow  water  dynamics  and  its  asymptotic accuracy is the same as for the well known Camassa-Holm equation\cite{DGH}. The DP equation was also proved formally integrable by constructing a Lax pair \cite{DHH} and the direct and inverse scattering approach to pursue it can be seen in \cite{CIL}. Moreover, they also presented \cite{DHH} that the DP equation has a bi-Hamiltonian structure and an infinite number of conservation laws, and admits travelling wave and exact peakon solutions which are analogous to the Camassa–Holm peakons\cite{Lenells,VP}.
The Cauchy problem of the DP equation is locally well-posed in certain Sobolev and Besov spaces \cite{GL,Y1,Y2}, also the non-uniform dependence and ill-posedness problem have recently been studied in \cite{HHG,LYZ1}. It has global strong solutions \cite{LY1,Y2}, the finite-time blow-up solutions \cite{ELY1,ELY2} and global weak solutions \cite{CK,ELY1,Y3}. 
Different form the Camassa-Holm equation,  the DP equation has not only peakon solutions \cite{DHH},
periodic peakon solutions \cite{Y3}, but also  shock peakons\cite{Lun} and the periodic  shock waves \cite{ELY2}.
Both Camassa-Holm and DP equation admit cusped travelling waves \cite{JL,Lenells}. This type of weak solutions are known to arise as solutions for the governing equations for waves in a channel or along a sloped beach \cite{Co35,HE1}, and also as solutions for the governing equations for equatorial ocean waves \cite{Co5,HE2}.

Recently, a lot of literature was devoted to studying the well-posedness (especially the ill-posedness) problem of the Camassa-Holm type equations (CH, DP and Novikov equation etc.) in the critical Besov spaces. For example, Guo et al. \cite{GLMY} prove norm inflation and hence ill-posedness for the Camassa-Holm type equations in the critical Sobolev space $H^\frac32$ and even in the Besov space $B^{1+\frac1p}_{p,r}(\mathbb{R})$ with $p\in [1,\infty], r\in (1,\infty],$ which implies $B^{1+\frac1p}_{p,1}(\mathbb{R})$ is the critical Besov space for these Camassa-Holm type equations. Li et al. \cite{LYZ2,LYZ3} demonstrated the non-continuity and then sharp ill-posedness of the Camassa-Holm type equations in $B^{\sigma}_{p,\infty}(\mathbb{R})$ with $\sigma>\max\{\frac32,1+\frac1p\}$. Later, it has been proved the local well-posedness for the Cauchy problem of the Camassa-Holm type equations in $B^{1+\frac1p}_{p,1}(\mathbb{R})$ with $p\in[1,\infty)$ by adopt the compactness argument and Lagrangian coordinate transformation\cite{YYG}. In the remaining case $p=\infty$, Guo et al. \cite{GYY}  prove the ill-posedness for the Camassa-Holm equation in $B^1_{\infty,1}.$

Different from the CH equation, the authors in \cite{LYZ} have proved the well-posedness for the Cauchy problem of the DP equation in critical Besov space $B^1_{\infty,1}.$ For the two component DP equation(\ref{dp}), the local well-posedness problem in certain Sobolev and Besov spaces has been studied by Yan et al.in \cite{YY}, here we mainly focus on the well-posedness problem in critical Besov space $B^1_{\infty,1}(\mathbb R)$. As the structure is changed from the DP equation due to the coupled density function $\rho$, here we have a different result for the two component DP equation (\ref{dp}) in this paper. Now, let us state our main theorem of this paper.

\begin{theorem}\label{th1}
For any $n\in \mathbb{Z}^+$ large enough there exists $(u_0,\rho_0)$ with
$$\|u_0\|_{B^1_{\infty,1}}+\|\rho_0\|_{B^0_{\infty,1}}\leq \frac{1}{\log\log n},$$
such that if we denote by $(u,\rho)\in \mathcal{C}([0,1];H^{3}\times H^2)$, the solution of the two component Degasperis-Procesi with initial data $(u_0,\rho_0)$, then
$$\|u(t_0)\|_{B^1_{\infty,1}}+\|\rho(t_0)\|_{B^0_{\infty,1}}\geq {\log\log n},$$
with $t_0\in (0,\frac{1}{\log n}]$.
\end{theorem}

\begin{remark}
    This implies the ill-posedness of the two component DP equation (\ref{dp}) in Besov space $B^1_{\infty,1}(\mathbb R)$. In fact, by Theorem \ref{th1} we have construct the norm inflation in space $B^1_{\infty,1}(\mathbb R)$, thus the data-to-solution map is not continuous at origin in this space, in the sense of Hadamard this implies the ill-posedness. 
  \end{remark}

\section{Littlewood-Paley analysis}\label{sec2}
Next, we will recall some facts about the Littlewood-Paley decomposition, the nonhomogeneous Besov spaces and their some useful properties (see \cite{B} for more details).

Let $\mathcal{B}:=\{\xi\in\mathbb{R}:|\xi|\leq 4/3\}$ and $\mathcal{C}:=\{\xi\in\mathbb{R}:3/4\leq|\xi|\leq 8/3\}.$
Choose a radial, non-negative, smooth function $\chi:\R\mapsto [0,1]$ such that it is supported in $\mathcal{B}$ and $\chi\equiv1$ for $|\xi|\leq3/4$. Setting $\varphi(\xi):=\chi(\xi/2)-\chi(\xi)$, then we deduce that $\varphi$ is supported in $\mathcal{C}$. Moreover,
\begin{eqnarray*}
\chi(\xi)+\sum_{j\geq0}\varphi(2^{-j}\xi)=1 \quad \mbox{ for any } \xi\in \R.
\end{eqnarray*}
We should emphasize that the fact $\varphi(\xi)\equiv 1$ for $4/3\leq |\xi|\leq 3/2$ will be used in the sequel.

For every $u\in \mathcal{S'}(\mathbb{R})$, the inhomogeneous dyadic blocks ${\Delta}_j$ are defined as follows
\begin{equation*}
{\Delta_ju=}
\begin{cases}
0,   &if \quad j\leq-2;\\
\chi(D)u=\mathcal{F}^{-1}(\chi \mathcal{F}u),  &if \quad j=-1;\\
\varphi(2^{-j}D)u=\mathcal{F}^{-1}\g(\varphi(2^{-j}\cdot)\mathcal{F}u\g), &if  \quad j\geq0.
\end{cases}
\end{equation*}
In the inhomogeneous case, the following Littlewood-Paley decomposition makes sense
$$
u=\sum_{j\geq-1}{\Delta}_ju\quad \text{for any}\;u\in \mathcal{S'}(\mathbb{R}).
$$
\begin{definition}[see \cite{B}]
Let $s\in\mathbb{R}$ and $(p,r)\in[1, \infty]^2$. The nonhomogeneous Besov space $B^{s}_{p,r}(\R)$ is defined by
\begin{align*}
B^{s}_{p,r}(\R):=\Big\{f\in \mathcal{S}'(\R):\;\|f\|_{B^{s}_{p,r}(\mathbb{R})}<\infty\Big\},
\end{align*}
where
\begin{numcases}{\|f\|_{B^{s}_{p,r}(\mathbb{R})}=}
\left(\sum_{j\geq-1}2^{sjr}\|\Delta_jf\|^r_{L^p(\mathbb{R})}\right)^{\fr1r}, &if $1\leq r<\infty$,\nonumber\\
\sup_{j\geq-1}2^{sj}\|\Delta_jf\|_{L^p(\mathbb{R})}, &if $r=\infty$.\nonumber
\end{numcases}
\end{definition}
The following Bernstein's inequalities will be used in the sequel.
\begin{lemma}[see Lemma 2.1 in \cite{B}] \label{lem2.1} Let $\mathcal{B}$ be a Ball and $\mathcal{C}$ be an annulus. There exist constants $C>0$ such that for all $k\in \mathbb{N}\cup \{0\}$, any positive real number $\lambda$ and any function $f\in L^p$ with $1\leq p \leq q \leq \infty$, we have
\begin{align*}
&{\rm{supp}}\hat{f}\subset \lambda \mathcal{B}\;\Rightarrow\; \|\pa_x^kf\|_{L^q}\leq C^{k+1}\lambda^{k+(\frac{1}{p}-\frac{1}{q})}\|f\|_{L^p},  \\
&{\rm{supp}}\hat{f}\subset \lambda \mathcal{C}\;\Rightarrow\; C^{-k-1}\lambda^k\|f\|_{L^p} \leq \|\pa_x^kf\|_{L^p} \leq C^{k+1}\lambda^k\|f\|_{L^p}.
\end{align*}
\end{lemma}
\begin{lemma}[see Lemma 2.100 in \cite{B}] \label{lem2.2}
Let $1 \leq r \leq \infty$, $1 \leq p \leq p_{1} \leq \infty$ and $\frac{1}{p_{2}}=\frac{1}{p}-\frac{1}{p_{1}}$. There exists a constant $C$ depending continuously on $p,p_1$, such that
$$
\left\|\left(2^{j}\left\|[\Delta_{j},v \pa_x] f\right\|_{L^{p}}\right)_{j}\right\|_{\ell^{r}} \leq C\left(\|\pa_x v\|_{L^{\infty}}\|f\|_{B_{p, r}^{1}}+\|\pa_x f\|_{L^{p_{2}}}\|\pa_x v\|_{B_{p_1,r}^{0}}\right).
$$
\end{lemma}

\section{Proof of Theorem \ref{th1}}

For localization, we first introduce the following bump function in the frequency space.
let $\widehat{\theta}\in \mathcal{C}^\infty_0(\mathbb{R})$ be a smooth and even function with values in $[0, 1]$ which satisfies
\bbal
\widehat{\theta}(\xi)=
\bca
1, \quad \mathrm{if} \ |\xi|\leq \frac{1}{200},\\
0, \quad \mathrm{if} \ |\xi|\geq \frac{1}{100}.
\eca
\end{align*}
Then, shift the frequency of the above function by letting $\phi(x)=\theta(x)
\sin(\frac{17}{24}x)$. For the convenience of notation, hereafter
\bbal
n\in 16\mathbb{N}^+=\{16,32,48,\cdots\},~~~~\mathbb{N}(n)=\{k\in \mathbb{N}^+:\frac{n}{4}\leq k\leq\frac{n}{2}\}.
\end{align*}
and define the norm $B^0_{\infty,1}\left(\mathbb{N}(n)\right)$ by
\bbal
\left\|u\right\|_{B^0_{\infty,1}(\mathbb{N}(n))}= \sum_{j\in \mathbb{N}(n)}\left\|\De_j u\right\|_{L^\infty}.
\end{align*}
Now, we construct the initial data $(u_0,\rho_0)$ as following:
\bbal
u_0(x)=0, \qquad \rho_0(x)&=n^{-\frac{1}{2}}\log n\cdot\sin(\frac{17}{12}2^nx)\cdot\sum_{\ell\in \mathbb{N}(n)}
\phi(2^\ell(x-2^{2n+\ell})).
\end{align*}

\subsection{Estimation of initial data}
First of all, we give the estimates for $\rho_0$.
\begin{lemma}\label{le-e1}
There exists a positive constant $C$ independent of $n$ such that
\bbal
&\|\rho_0\|_{L^\infty}\leq C n^{-\frac12}\log n,\\
&\|\rho_0\|_{B^0_{\infty,1}}\leq C\|\rho_0\|_{L^\infty}\leq C n^{-\frac12}\log n.
\end{align*}
\end{lemma}
\begin{proof} By the definition of $\rho_0$ and using convolution expressions, it's easy to verify that the support set of $\rho_0$ satisfies
\bal\label{spr}
\mathrm{supp} \ \widehat{\rho_0}\subset \left\{\xi\in\R: \ \frac{17}{12}2^{n}-\frac{17}{12}2^{\ell}-\fr12\leq |\xi|\leq \frac{17}{12}2^{n}+\frac{17}{12}2^{\ell}+\fr12\right\}.
\end{align}
Since $\phi$ and $\check{\chi}$  are Schwartz functions, we have
\bal\label{phi-es}
|\phi(x)|+|\check{\chi}(x)|\leq C(1+|x|)^{-M}, \qquad  M\geq 100.
\end{align}
then, by a direct computation, we have
\bbal
\|\rho_0\|_{L^\infty}&\leq Cn^{-\frac{1}{2}}\log n\left\|\sum_{\ell\in \mathbb{N}(n)}\phi(2^\ell(x-2^{2n+\ell}))\right\|_{L^\infty}
\\&\leq Cn^{-\frac{1}{2}}\log n\left\|\sum_{\ell\in \mathbb{N}(n)}\frac{1}{(1+|2^\ell x-2^{2n+2\ell}|)^M}\right\|_{L^\infty}\leq Cn^{-\frac{1}{2}}\log n.
\end{align*}
Due to the fact of (\ref*{spr}) and the definition of $\Delta_j$, it's easy to check that $\Delta_j\rho_0=\rho_0$ if $j=n$ and $\Delta_ju_0=0$ if $j\neq n$, we deduce that
\bbal
\|\rho_0\|_{B^0_{\infty,1}}&\leq C\|\rho_0\|_{L^\infty}\leq C n^{-\frac12}\log n.
\end{align*}
This completes the proof of Lemma \ref{le-e1}.
\end{proof}

The following lower bound estimate of the squared term is crucial to our proof.
\begin{lemma}\label{le-e2}
There exists a positive constant $c$ independent of $n$ such that
\bbal
\left\|(\rho_0)^2\right\|_{B^0_{\infty,1}\left(\mathbb{N}(n)\right)}\geq c\log^2n, \qquad n\gg1.
\end{align*}
\end{lemma}
\begin{proof} It follows that for any $j\in \mathbb{N}(n)$,
\bal\label{l}
\De_j[(\rho_0)^2]=\frac12n^{-1}\log^2n
\cdot\De_jU, \qquad U:=\left(\sum_{\ell\in \mathbb{N}(n)}
\phi(2^\ell(x-2^{2n+\ell}))\right)^2.
\end{align}
We rewrite $U$ as
\bbal
U=\sum_{\ell\in \mathbb{N}(n)}
\phi^2(2^\ell(x-2^{2n+\ell}))+\sum_{\ell,m\in \mathbb{N}(n)\setminus\{\ell=m\}}
\phi(2^\ell(x-2^{2n+\ell}))\phi(2^m(x-2^{2n+m})):=U_1+U_2.
\end{align*}
Then, we have $\De_j U=\De_j U_1+\De_j U_2$. For $\De_jU_2$, we have
\bal\label{y1}
\begin{split}
||\De_jU_2||_{L^\infty}&\leq C||U_2||_{L^\infty}
\\&\leq C\sum_{\ell,m\in \mathbb{N}(n)\setminus\{\ell=m\}}||\phi(2^\ell(x-2^{2n+\ell}))\phi(2^m(x-2^{2n+m}))||_{L^\infty}
\\&\leq C\sum_{\ell,m\in \mathbb{N}(n)\setminus\{\ell=m\}}||\frac{1}{(1+2^\ell|x-2^{2n+\ell}|)^M}\cdot \frac{1}{(1+2^m|x-2^{2n+m}|)^M}||_{L^\infty}
\\&\leq C\sum_{\ell,m\in \mathbb{N}(n)\setminus\{\ell=m\}}2^{-2nM}\leq Cn^22^{-2nM}.
\end{split}
\end{align}
By direct computations, for large enough $n$, we have
\bbal
\Delta_jU_1=\Delta_j\phi^2(2^j(x-2^{2n+j}))+\Delta_j\sum_{j\neq \ell\in \mathbb{N}(n)}\phi^2(2^\ell(x-2^{2n+\ell}))=:\Delta_jU_{1,1}+\Delta_jU_{1,2}.
\end{align*}
Let us introduce the set $\mathbf{B}_j$ defined by $\mathbf{B}_j=\{x:2^j|x-2^{2n+j}|\leq 1\}$.

We can show that
\bbal
||\Delta_j U_1||_{L^\infty(\mathbf{B}_j)}\geq ||\Delta_jU_{1,1}||_{L^\infty(\mathbf{B}_j)}-||\Delta_jU_{1,2}||_{L^\infty(\mathbf{B}_j)}
\end{align*}
Notice that
\bbal
\phi^2(x)=\frac12\theta^2(x)-\frac12\theta^2(x)\cos(\frac{17}{12}x):=\Phi_1(x)+\Phi_2(x),
\end{align*}
then we have
\bbal
\Delta_jU_{1,1}=\Delta_j\Phi_2(2^j(x-2^{2n+j}))=\Phi_2(2^j(x-2^{2n+j})).
\end{align*}
This implies
\bal\label{y2}
||\Delta_jU_{1,1}||_{L^\infty(\mathbf{B}_j)}\geq \frac12\theta^2(0):= c.
\end{align}
By \eqref{phi-es}, we have
\bbal
\|\Delta_jU_{1,2}\|_{L^\infty(\mathbf{B}_j)}&\leq C\sum_{j\neq \ell\in \mathbb{N}(n)}\left\|2^j\int_{\R}
\check{\chi}(2^j(x-y))\phi^2(2^\ell(y-2^{2n+\ell}))
\dd y\right\|_{L^\infty(\mathbf{B}_j)}
\\&\leq C\sum_{j\neq \ell\in \mathbb{N}(n)}\left\|2^j\int_{\R}
\frac{1}{(1+2^j|x-y|)^{M}}\frac{1}{(1+2^\ell|y-2^{2n+\ell}|)^{2M}}
\dd y\right\|_{L^\infty(\mathbf{B}_j)}.
\end{align*}

Dividing the integral region in terms of $y$ into the following two parts to estimate:
\begin{align*}
\mathbf{A}_{1} &:=\left\{y :\; |y-2^{j+2 n} |\leq 2^{2 n}\right\} , \\
\mathbf{A}_{2} &:=\left\{y :\; |y-2^{j+2 n} |\geq 2^{2 n} \right\}.
\end{align*}
For $x \in \mathbf{B}_{j}$ and $y \in \mathbf{A}_{1}$, it is easy to check that
\begin{align*}
\left|y-2^{\ell+2 n}  \right| &=\left|y-2^{j+2 n}  +2^{j+2 n} -2^{\ell+2 n}  \right| \\
& \geq\left|2^{\ell+2 n}  -2^{j+2 n}  \right|-\left|y-2^{j+2 n} \right| \geq 2^{2 n}.
\end{align*}
For $x \in \mathbf{B}_{j}$ and $y \in \mathbf{A}_{2}$, it is easy to check that
\begin{align*}
|x-y| &=\left|x-2^{j+2 n}  +2^{j+2 n} -y\right| \\
& \geq\left|y -2^{j+2 n}  \right|-\left|x-2^{j+2 n} \right| \geq  2^{2 n}-2^{-j}\geq 2^{2 n-1}.
\end{align*}
Then, we have
\bbal
\left\|2^j\int_{\R}
\frac{1}{(1+2^j|x-y|)^{M}}\frac{1}{(1+2^\ell|y-2^{2n+\ell}|)^{2M}}
\dd y\right\|_{L^\infty(\mathbf{B}_j)}\leq C2^j2^{-2nM}\leq C2^{\frac n2}2^{-nM},
\end{align*}
which implies
\bal\label{y3}
\|\Delta_jU_{1,2}\|_{L^\infty(\mathbf{B}_j)}&\leq C2^{-n}
\end{align}
Combining \eqref{y1}-\eqref{y3}, we deduce that for $n\gg1$,
\bbal
\left\|\De_j[(\rho_0)^2]\right\|_{L^\infty}&\geq \left\|\De_j[(\rho_0)^2]\right\|_{L^\infty(\mathbf{B}_j)}
\\&\geq  \left\|\Delta_jU_{1,1}\right\|_{L^\infty(\mathbf{B}_j)}
-\left\|\Delta_jU_{1,2}\right\|_{L^\infty(\mathbf{B}_j)}
-\left\|\De_jU_{2}\right\|_{L^\infty}
\\&\geq cn^{-1}\log^2n.
\end{align*}
Therefore, by the definition of the Besov norm, we have
\bbal
\left\|(\rho_0)^2\right\|_{B^0_{\infty,1}(\mathbb{N}(n))}= \sum_{j\in \mathbb{N}(n)}\left\|\De_j[(\rho_0)^2]\right\|_{L^\infty}\geq c\log^2n.
\end{align*}
This completes the proof of lemma \ref{le-e2}.
\end{proof}

\subsection{Norm Inflation}
Firstly, we define the Lagrangian flow-map $\psi$ associated to $u$ by solve the following ODE:
\begin{equation}\label{flow}
    \left\{
    \begin{array}{ll}
    \frac{\mathrm d}{\mathrm d t}\psi(t,x)=u(t,\psi(t,x)), \\[1ex]
    \psi(0,x)=x.
    \end{array}
    \right.
    \end{equation}
Considering the equation
\begin{equation}\label{vfl}
    \left\{
    \begin{array}{ll}
    \pa_t v+u\partial_xv=P,  &t \in [0,T), \quad x \in \mathbb{R},\\[1ex]
    v(0,x)=x, x \in \mathbb{R}.
    \end{array}
    \right.
    \end{equation}
Taking $\De_j$ on \eqref{vfl}, we get
$$\pa_t (\De_jv)+u\partial_x\De_jv=[u,\De_j]\pa_xv+\De_jP:=R_j+\De_jP,$$
Due to (\ref{flow}), then
$$\frac{\mathrm d}{\mathrm d t}\big((\De_jv)\circ\psi\big)=R_j\circ\psi+\De_jP\circ\psi,$$
which is equivalent to the integral form
\bal\label{l6}
(\De_jv)\circ \psi&=\De_jv_0+\int^t_0R_j\circ \psi\dd \tau +\int^t_0\De_jP\circ \psi\dd \tau.
\end{align}

Now we can go back to the proof of the Theorem \ref{th1}. For $n\gg1$, it is easy to check that $(u,\rho)\in \mathcal{C}([0,1];H^{3}\times H^2)$ which satisfy for $t\in[0,1]$
\begin{align*}
&\|u(t)\|_{W^{1,\infty}}+||\rho(t)||_{L^\infty}\leq (\|u_0\|_{W^{1,\infty}}+||\rho_0||_{L^\infty})\exp\left(C\int_0^t \|\pa_xu(\tau),\rho(\tau)\|_{L^{\infty}}\mathrm{d}\tau\right).
\end{align*} 
For $n\gg1$, we have for $t\in[0,1]$
\bal\label{u,rho-es}
\|u\|_{W^{1,\infty}}+\|\rho\|_{L^\infty}\leq C\|\rho_0\|_{L^\infty}\leq Cn^{-\frac12}\log n.
\end{align}
Moreover, we also have  $t\in[0,1]$
\bal\label{u-es}
\|u\|_{L^\infty}&\leq C\int^t_0(\|u\|^2_{W^{1,\infty}}+||\rho||^2_{L^\infty})\dd \tau
\leq Cn^{-1}\log^2 n.
\end{align}
To prove Theorem \ref{th1}, it suffices to show that there exists $t_0\in(0,\frac{1}{\log n}]$ such that
\bal\label{m}
\|u(t_0,\cdot)\|_{B^1_{\infty,1}}\geq \log\log n.
\end{align}
We prove \eqref{m} by contraction. If \eqref{m} were not true, then
\bal\label{m1}
\sup_{t\in(0,\frac{1}{\log n}]}\|u(t,\cdot)\|_{B^1_{\infty,1}}< \log\log n.
\end{align}
We divide the proof into two steps.

{\bf Step 1: Lower bounds for $(\De_ju)\circ \psi$}

Now we consider the equation along the Lagrangian flow-map associated to $u$.
Utilizing \eqref{l6} to \eqref{dp} yields
\bbal
(\De_ju)\circ \psi&=\De_ju_0+\int^t_0R^1_j\circ \psi\dd \tau +\int^t_0\De_jF\circ \psi\dd \tau+\int^t_0\big(\De_jE\circ \psi-\De_jE_0\big)\dd \tau+t\De_jE_0,
\end{align*}
where
\bbal
&R^1_j=[u,\De_j]\pa_xu,\quad F=-\frac32\pa_x(1-\pa^2_x)^{-1}(u^2), \quad E=-\frac12\pa_x(1-\pa^2_x)^{-1}(\rho^2).
\end{align*}
Due to Lemma \ref{le-e2}, we deduce
\bal\label{g1}
\sum_{j\in \mathbb{N}(n)}2^j\|\De_jE_0\|_{L^\infty}
\approx \sum_{j\in \mathbb{N}(n)}\|\De_j\pa_xE_0\|_{L^\infty}
\geq c\sum_{j\in \mathbb{N}(n)}\|\De_j(\rho_0)^2\|_{L^\infty}
\geq c\log^2n.
\end{align}
Notice that, for fixed $t$ the Lagrangian flow-map $\psi(t,\cdot)$ is a diffeomorphism of $\mathbb{R}$, then we have for $t\in(0,\frac{1}{\log n}]$,
\bbal
\|f(t,\psi(t,x))\|_{L^\infty}= \|f(t,\cdot)\|_{L^\infty}.
\end{align*}
Then, using \eqref{u,rho-es}, \eqref{m1} and the commutator estimate from Lemma \ref{lem2.2}, we have
\bal\label{g2}
\sum_{j\geq -1}2^j\|R^1_j\circ \psi\|_{L^\infty}&\leq C\sum_{j\geq -1}2^j\|R^1_j\|_{L^\infty}\leq C\|\pa_xu\|_{L^\infty}\|u\|_{B^1_{\infty,1}}\leq C n^{-\frac12}\log n \cdot\log\log n.
\end{align}
Also, by \eqref{u-es}, we have
\bbal
2^j\|\De_jF\circ \psi\|_{L^\infty}&\leq C2^j\|\De_jF\|_{L^\infty}
\leq C\|u\|^2_{L^\infty}
\leq Cn^{-2}\log^4n,
\end{align*}
which implies
\bal\label{g3}
\sum_{j\in \mathbb{N}(n)}2^j\|\De_jF\circ \psi\|_{L^p}\leq C n^{-1}\log^4n.
\end{align}
Combining \eqref{g1}-\eqref{g3} and using Lemmas \ref{le-e1}-\ref{le-e2} yields
\bbal
\sum_{j\in \mathbb{N}(n)}2^j\|(\De_ju)\circ \psi\|_{L^\infty}
&\geq t\sum_{j\in \mathbb{N}(n)}2^j\|\De_jE_0\|_{L^\infty}-\sum_{j\in \mathbb{N}(n)}2^j\|\De_jE\circ \psi-\De_jE_0\|_{L^\infty}-
Cn^{-1}\log^4n
\\&\geq ct\log^2n-\sum_{j\in \mathbb{N}(n)}2^j\|\De_jE\circ \psi-\De_jE_0\|_{L^\infty}-
Cn^{-1}\log^4 n.
\end{align*}
{\bf Step 2: Upper bounds for $\De_jE\circ \psi-\De_jE_0$}

By easy computations, we can see that
\bal\label{E}
\pa_tE+u\pa_xE
&=-\pa_x(1-\pa^2_x)^{-1}(\pa_t\rho \rho)-\frac12u\pa^2_x(1-\pa^2_x)^{-1}(\rho^2)
\\&=-\pa_x(1-\pa^2_x)^{-1}(-\frac12\pa_x(u\rho^2)
-\frac32u_x\rho^2)-\frac12u(1-\pa^2_x)^{-1}(\rho^2)+\frac12u\rho^2 \nonumber
\\&=\frac12\pa^2_x(1-\pa^2_x)^{-1}(u\rho^2)
+\frac32\pa_x(1-\pa^2_x)^{-1}(u_x\rho^2)-\frac12u(1-\pa^2_x)^{-1}(\rho^2)+\frac12u\rho^2 \nonumber
\\&=\frac12(1-\pa^2_x)^{-1}(u\rho^2)-\frac12u\rho^2
+\frac32\pa_x(1-\pa^2_x)^{-1}(u_x\rho^2)-\frac12u(1-\pa^2_x)^{-1}(\rho^2)+\frac12u\rho^2 \nonumber
\\&=:G\nonumber,
\end{align}
where
\bbal
G=&-\frac 12u(1-\pa^2_x)^{-1}(\rho^2)
-\frac 12(1-\pa^2_x)^{-1}\left(-
u\rho^2-\pa_x(3u_x\rho^2)\right).
\end{align*}
Utilizing \eqref{l6} to \eqref{E} yields
\bbal
\De_jE\circ \psi-\De_jE_0=\int^t_0[u,\De_j]\pa_xE\circ \psi\dd \tau +\int^t_0\De_jG\circ \psi\dd \tau.
\end{align*}
Using the commutator estimate from Lemma \ref{lem2.2}, one has
\bbal
2^j\|[u,\De_j]\pa_xE\|_{L^\infty}\leq C(\|\pa_xu\|_{L^\infty}\|E\|_{B^1_{\infty,\infty}}
+\|\pa_xE\|_{L^\infty}\|u\|_{B^1_{\infty,\infty}})\leq C(\|u\|_{W^{1,\infty}}+\|\rho\|_{L^\infty})^3
\end{align*}
and
\bbal
2^j\|\De_jG\|_{L^\infty}\leq C(\|u\|_{W^{1,\infty}}+\|\rho\|_{L^\infty})^3.
\end{align*}
Then, we have from \eqref{u,rho-es}
\bbal
2^j\|\De_jE\circ \psi-\De_jE_0\|_{L^\infty}\leq C(\|u\|_{W^{1,\infty}}+\|\rho\|_{L^\infty})^3\leq Cn^{-\frac32}\log^3n,
\end{align*}
which leads to
\bbal
\sum_{j\in \mathbb{N}(n)}2^j\|\De_jE\circ \psi-\De_jE_0\|_{L^p}
\leq Cn^{-\frac12}\log^3n.
\end{align*}
Combining Step 1 and Step 2, then for $t=\frac{1}{\log n}$, we obtain for $n\gg1$
\bbal
\|u(t)\|_{B^1_{\infty,1}}&\geq \|u(t)\|_{B^1_{\infty,1}(\mathbb{N}(n))}\geq C\sum_{j\in \mathbb{N}(n)}2^j\|(\De_ju)\circ \psi\|_{L^\infty}
\\&\geq c t\log^2n-Cn^{-\frac12}\log^3 n\geq \log\log n,
\end{align*}
which contradicts the hypothesis \eqref{m1}.

In conclusion, we obtain the norm inflation and hence the ill-posedness of the two component  Degasperis-Procesi equation. Thus, Theorem \ref{th1} is proved.

\section*{Acknowledgments} J. Li is supported by the National Natural Science Foundation of China (11801090 and 12161004) and Jiangxi Provincial Natural Science Foundation, China (20212BAB211004). M. Li was supported by Educational Commission Science Programm of Jiangxi Province (No. GJJ190284) and Natural Science Foundation of Jiangxi Province (No. 20212BAB211011). W. Zhu is supported by the Guangdong Basic and Applied Basic Research Foundation, China (2021A1515111018).

\end{document}